\documentclass{article}

\usepackage{amssymb,amsmath,xcolor,graphicx,xspace,colortbl,rotating} 
\usepackage{graphics}  
\usepackage{ragged2e}  
\usepackage{tabulary}  
\usepackage{varioref}  
\usepackage{xcolor}  
\graphicspath{{NilpotentB-OrbitsinB2_graphics/}{NilpotentB-OrbitsinB2_tcache/}{NilpotentB-OrbitsinB2_gcache/}}
\DeclareGraphicsExtensions{.pdf,.eps,.ps,.png,.jpg,.jpeg}\newtheorem {theorem}{Theorem}

 \newtheorem {lemma}[theorem]{Lemma}
 
\newtheorem {proposition}[theorem]{Proposition} 
 
\newenvironment {proof}[1][Proof]{\noindent \textbf {#1.} }{\ \rule {0.5em}{0.5em}}
\begin{document}
\title{Nilpotent Orbits for Borel Subgroups of
$SO_{5}(k)$}

\author{ Madeleine Burkhart and David Vella
}
\date{August 16, 2017
}
\maketitle

\begin{abstract}Let
$G$
be a quasi-simple algebraic group defined over an algebraically closed field
$k$
and
$B$
a Borel subgroup of
$G$
acting on the nilradical
$\mathfrak{n}$
of its Lie algebra
$\mathfrak{b}$
via the Adjoint representation.  It is known that
$B$
has only finitely many orbits in only five cases: when
$G$
is of type
$A_{n}$
for
$n \leq 4 ,$
and
$G$
is type
$B_{2} .$
In this paper, we elaborate on this work in the case when
$G =SO_{5}(k)$
(type
$B_{2})$
by finding the defining equations of each orbit.  We use these equations to determine the
dimension of the orbits and the closure ordering on the set of orbits. The other four cases, when
$G$
is type
$A_{n} ,$
can be approached the same way and are treated in a separate paper.

\end{abstract}

\section{ Introduction.
}
Before specializing to
$G =SO_{5}(k) ,$
we make some general remarks in order to provide context and some motivation for our work.
Let
$k$
be an algebraically closed field and
$G$
a quasi-simple algebraic group over
$k .$
Fix a maximal torus
$T$
of
$G$, and let
$\Phi $
denote the root system of
$G$
relative to
$T$
($\Phi $
is irreducible since
$G$
is quasi-simple.)
Fix a set \/$\Delta $
of simple roots in
$\Phi  ,$
with corresponding set of positive roots
$\Phi ^{ +} ,$
and let
$B =TU$
($U$
is the unipotent radical of
$B)$
be the Borel subgroup of
$G$
determined by
$\Phi ^{ +} .$
Write the one-dimensional unipotent root group corresponding to a root
$\alpha $
as
$U_{\alpha } .$
Denote the Lie algebra of
$G$
by
$\mathfrak{g}$
, that of
$T$
by
$\mathfrak{h}$
, and that of
$B$
by
$\mathfrak{b} .$
Then the nilradical
$\mathfrak{n} =\mathfrak{n}(\mathfrak{b})$
of
$\mathfrak{b}$
is in fact the Lie algebra of
$U$, and we have decompositions
$\mathfrak{b} =\mathfrak{h} \oplus \mathfrak{n} ,$
and
$\mathfrak{n} = \oplus _{\alpha  \in \Phi ^{ +}}\mathfrak{g}_{\alpha }$
as vector spaces, where
$\mathfrak{g}_{\alpha }$
is the root space of
$\mathfrak{g}$
corresponding to
$\alpha  ,$
and is also the Lie algebra of
$U_{\alpha } .$

$G$
acts on
$\mathfrak{g}$
via the Adjoint representation, and the orbits of this action have been intensely studied,
partly because there are connections between the nilpotent orbit theory and the
representation theory of
$G .$
It is known that there are only finitely many nilpotent
$G$-orbits (a \emph{nilpotent orbit} means an orbit of a nilpotent element of
$\mathfrak{g} .)$
There are combinatorial indexing sets for these nilpotent orbits, and there are formulas to
compute the dimension of each orbit.  Also, it is known which orbits are in the Zariski closures of
any given orbit (the \emph{closure ordering}.) Therefore, it is well understood
how all the nilpotent orbits fit together to form a larger object, called the
\emph{nullcone}
$\mathcal{N}$
of
$\mathfrak{g} ,$
which is the union of the nilpotent orbits.  For details of this classical theory, see [CM]
for the characteristic zero case and [C] or [J] more generally.

The notion of a support variety of a module is one example of a connection between
nilpotent
$G$-orbits and representation theory, when the characteristic of
$k$
is
$p >0.$
In this case recall that there is a
$p$\emph{-th power map}
$x \mapsto x^{[p]}$
on
$\mathfrak{g}$
that makes
$\mathfrak{g}$
into a restricted Lie algebra, and there is a \emph{Frobenius map}
$F :G \rightarrow G$
whose kernel is denoted
$G_{1}$
which is an infinitesimal group scheme, whose rational representation theory coincides with
the representation theory of
$\mathfrak{g} .$
Similarly, denote the kernel of
$F :B \rightarrow B$
by
$B_{1} .$
By results of Friedlander and Parshall [FP], the \emph{cohomology variety
}of
$G_{1}$
(the maximal ideal spectrum of the cohomology ring
$H^{ \ast }(G_{1} ,k)$
identifies naturally with a subvariety
$\mathcal{N}_{1} =\{x \in \mathcal{N}$
\textbar{}
$x^{[p]} =0\}$
of the nullcone
$\mathcal{N}$
of
$\mathfrak{g} ,$
and furthermore, for any finite-dimensional
$\mathfrak{g}$-module
$M$
, there is an important subvariety of
$\mathcal{N}_{1}$
called the \emph{support variety} of
$M$
, denoted
$V_{\mathfrak{g}}(M)$
or
$V_{G_{1}}(M)$
. If
$M$
is a rational
$G$-module, then
$V_{G_{1}}(M)$
is
$G$-stable in
$\mathcal{N}_{1}$
and is therefore a union of nilpotent
$G$-orbits.

 Aspects of the representation theory of
$G_{1}$
\thinspace are determined by this support variety (for example,
$M$
is a projective module if and only if
$V_{G_{1}}(M) =\{0\}$), so one would like to be able to compute these support varieties, and knowing they are
unions of nilpotent
$G$-orbits may be useful.

 If
$H$
is a closed subgroup of
$G$
and
$M$
is a rational
$H$-module, denote the rational
$G$-module induced from
$M$
by
$M\vert _{H}^{G} .$
Now let
$X(T)$
be the character group of
$T ,$
and for
$\lambda  \in X(T) ,$
we also use the symbol
$\lambda $
to denote the one-dimensional
$T$-module on which
$T$
acts via
$\lambda  :$
$t .v =\lambda (t)v$
for all
$t \in T .$
This rational
$T$-module extends to a rational
$B$-module by trivial
$U$
action, also denoted
$\lambda  .$
The modules
$\lambda \vert _{B}^{G}$
are (the duals of) the well-known \emph{Weyl modules} of
$G .$
Another important class of modules are those of the form
$Z(\lambda ) =\lambda \vert _{B_{1}}^{G_{1}} ,$
sometimes called \emph{baby Verma modules}. (The name comes from the fact
that there is an alternate definition of
$Z(\lambda )$
which is analogous to the definition of a Verma module for
$\mathfrak{g} ,$
while the adjective `baby' can be interpreted a alluding to the infinitesimal subgroups in our definition using
induction, or to the fact that the $Z(\lambda )$ are finite dimensional and the usual Verma modules are not.)

 One of the goals of the paper [NPV] was to calculate the support varieties
$V_{G_{1}}(\lambda \vert _{B}^{G})$
in order to prove the ``Jantzen conjecture'' [NPV, theorem 6.21.]  A central idea of [NPV]
is to compare
$V_{G_{1}}(\lambda \vert _{B}^{G})$
to
$V_{G_{1}}(Z(\lambda )) .$
Of course,
$Z(\lambda )$
is not a
$G$-module (only a
$G_{1}$-module), so one would not expect
$V_{G_{1}}(Z(\lambda ))$
to be a
$G$-stable subset of
$\mathcal{N}_{1} .$
However, it turns out that this variety is stable under the action of
$B$
and therefore
$V_{G_{1}}(Z(\lambda ))$
is a union of nilpotent
$B$-orbits [NPV, Prop.7.1.1].

 That nilpotent
$B$-orbits (as well as nilpotent
$G$-orbits) have a connection to the representation theory of
$G$
provides a motivation for studying nilpotent
$B$-orbits. Even without this explicit motivation, it is interesting to to try to generalize
what is known about nilpotent
$G$-orbits to nilpotent
$B$-orbits. However, nilpotent
$B$-orbits are not nearly as tidy as is the case for nilpotent
$G$-orbits. One important difference is that most of the time there are infinitely many
nilpotent
$B$-orbits. Thus, finding a nice indexing set for the orbits could be difficult.  In general,
we would expect an indexing set to have a continuous piece as well as a discrete piece, as certain
infinite families of orbits might be described by continuous parameters.

 In this paper, our focus is on the case
$G =SO_{5}(k) ,$
which is one of the five cases where there are finitely many nilpotent
$B$-orbits. For each of the orbits, we find the defining polynomial equations of the
orbit.
From these calculations, it is easy to to determine both the dimension of each orbit
as well as the closure ordering for the set of orbits. \bigskip

\section{ Nilpotent
$B$-Orbits in
$SO_{5}(k)$.
}
Throughout this section we assume the characteristic of
$k$
is either
$0$
or a prime
$p \neq 2.$
The following Proposition is a basic fact about algebraic group actions. A proof can be
found in [B] or [H2].

\begin{proposition}
 Let
$G$
be an algebraic Group acting morphically on a non-empty variety
$V .$
Then each orbit is a locally closed, smooth variety, and its boundary is a union of
orbits of strictly lower dimension.
\end{proposition}

 Thus, the orbit
$G .x$
is open and dense in its closure
$\overline{G .x}$, and hence has the same dimension as its closure. 

We study the action of a Borel subgroup
$B$
of
$G =SO_{5}(k)$
on the nilradical
$\mathfrak{n}$
of it's Lie algebra
$\mathfrak{b} \subseteq \mathfrak{s}\mathfrak{o}_{5}(k) ,$
via the Adjoint representation. Our main results consist of finding the defining
equations of each nilpotent
$B$-orbit, which will exhibit each orbit explicitly as an intersection of an open set and a
closed set.  From there it will be an easy matter to determine the closure of each orbit, and
thereby find the dimension of each orbit as well as to determine which orbits comprise the boundary
of a given orbit.  That is, we will determine the partial order determined by the orbit closures,
which is defined as
$G .x \preceq G .y$
if and only if
$G .x \subseteq \overline{G .y} .$

Let
$f$
be a polynomial.  We use the standard notation that the zero set of
$f$
is written as
$Z(f)$
and that
$Z(f ,g) =Z(f) \cap Z(g)$
is the set of common zeros of polynomials
$f$
and
$g .$
Then if we have a finite set of polynomials then
$Z(f_{1} ,f_{2} , . . . ,f_{k})$
is a Zariski-closed set, that is, it is an affine variety.  The notation
$V(f)$
denotes the complement of
$Z(f)$
- the set of non-zeros of
$f ,$
and so
$V(f)$
is a Zariski open set.  A locally closed set is an intersection of an open set and a closed
set, and in this section the orbits will turn out to be locally closed sets of the form
$V =Z(f_{1} ,f_{2} , . . . ,f_{k}) \cap V(g_{1}) \cap V(g_{2}) \cap  . . . \cap V(g_{\ell })$
for polynomials
$f_{i}$
and
$g_{j} .$
Observe that the closure of
$V$
is then
$Z(f_{1} ,f_{2} , . . . ,f_{k})$
and
$V$
is open and dense in this closure, whence
$\dim V =\dim Z(f_{1} ,f_{2} , . . . ,f_{k}) .$

If
$U\gamma $
is a root group of
$G ,$
then
$U_{\gamma }(t)$
denotes the image of
$t$
under the standard isomorphism
$k_{add} \approx U_{\gamma }$
. In classical groups, the Adjoint action on the Lie algebra is simply conjugation of
matrices. The matrix
$e_{ij}$
is the matrix with a
$1$
in the
$ij$
position and
$0$
everywhere else.         Now
$\mathfrak{g} =\mathfrak{s}\mathfrak{o}_{5}(k)$
, and we take
$T$
to be the diagonal matrices in
$G .$
More precisely, a typical element of the torus
$T$
has the form:
\label{Toruselement}
\begin{equation*}T(s ,t) =diag(1 ,s ,t ,s^{ -1} ,t^{ -1}) =\left [\begin{array}{ccccc}1 & 0 & 0 & 0 & 0 \\
0 & s & 0 & 0 & 0 \\
0 & 0 & t & 0 & 0 \\
0 & 0 & 0 & s^{ -1} & 0 \\
0 & 0 & 0 & 0 & t^{ -1}\end{array}\right ]
\end{equation*}

for
$s$
and
$t$
nonzero in
$k .$

 For root systems we use the standard notation one finds in [H1] or
[Bou].  In type
$B_{2} ,$
the simple roots are
$\Delta  =\{\alpha _{1} ,\alpha _{2}\}$
and
$\Phi ^{ +} =\{\alpha _{1} ,$
$\alpha _{2} ,$
$\alpha _{1} +\alpha _{2} ,$
$\alpha _{1} +2\alpha _{2}\} ,$
so
$\mathfrak{n}$
$ \approx k^{4}$
as an affine variety and vector space.  The root vectors in $\mathfrak{n}$ are the matrices:

\begin{center}
\begin{tabular}[c]{l}
$x_{\alpha _{1}} =e_{23} -e_{54}$
\\
$x_{\alpha _{2}} =e_{15} -e_{31}$
\\
$x_{\alpha _{1} +\alpha _{2}} =e_{14} -e_{21}$
\\
$x_{\alpha _{1} +2\alpha _{2}} =e_{25} -e_{34}$
\end{tabular}\end{center}\par
Now each root space has a coordinate function, which we denote by a capital
$X$
with the same subscript as the root space.  In other words, a typical element
of
$\mathfrak{n}$
has the form of a linear combination of the four root vectors:
$x =a_{1}x_{\alpha _{1}} +a_{2}x_{\alpha _{2}} +a_{3}x_{\alpha _{1} +\alpha _{2}} +a_{4}x_{\alpha _{1} +2\alpha _{2}} ,$
which we can write in coordinate form as
$(a_{1} ,a_{2} ,a_{3} ,a_{4})$
(showing explicitly
$\mathfrak{n} \approx k^{4})$
, and the coordinate function just selects the appropriate coordinate, so
$X_{\alpha _{1}}(x) =a_{1}$
and
$X_{\alpha _{1} +2\alpha _{2}}(x) =a_{4} ,$
for example.  Thus, the
$B$
orbits in which we are interested are locally closed sets in
$\mathfrak{n}$
or
$k^{4}$
which are defined by polynomials in the four variables of the polynomial ring
$k[X_{\alpha _{1}} ,$
$X_{\alpha _{2}} ,\thinspace X_{\alpha _{1} +\alpha _{2}} ,$
$X_{\alpha _{1} +2\alpha _{2}}]$
.

 To begin the calculations, let's determine the
$B$
orbit of the highest root vector $x_{\alpha _{1} +2\alpha _{2}}$
. Since the highest root vector is the high weight of the Adjoint representation, it is
a maximal vector - it is fixed by
$U ,$
the unipotent radical of
$B ,$
and sent to a multiple of itself by the torus
$T .$
It follows that
$B .x_{\alpha _{1} +2\alpha _{2}} =T .U .x_{\alpha _{1} +2\alpha _{2}} =T .x_{\alpha _{1} +2\alpha _{2}} \subseteq \mathfrak{g}_{\alpha _{1} +2\alpha _{2}}$
. Thus, we need only compute the
$T$
orbit of this weight vector, which is easy by direct calculation.  Abbreviating the root
vector by
$x$
, we have:
\begin{align*}T(s ,t) .x =T(s ,t)xT(s ,t)^{ -1} \\
 =\left [\begin{array}{ccccc}1 & 0 & 0 & 0 & 0 \\
0 & s & 0 & 0 & 0 \\
0 & 0 & t & 0 & 0 \\
0 & 0 & 0 & s^{ -1} & 0 \\
0 & 0 & 0 & 0 & t^{ -1}\end{array}\right ]\left [\begin{array}{ccccc}0 & 0 & 0 & 0 & 0 \\
0 & 0 & 0 & 0 & 1 \\
0 & 0 & 0 &  -1 & 0 \\
0 & 0 & 0 & 0 & 0 \\
0 & 0 & 0 & 0 & 0\end{array}\right ] \\
\begin{array}{c} \cdot \left [\begin{array}{ccccc}1 & 0 & 0 & 0 & 0 \\
0 & s^{ -1} & 0 & 0 & 0 \\
0 & 0 & t^{ -1} & 0 & 0 \\
0 & 0 & 0 & s & 0 \\
0 & 0 & 0 & 0 & t\end{array}\right ]\end{array}\end{align*}

\begin{equation*} =\left [\begin{array}{ccccc}0 & 0 & 0 & 0 & 0 \\
0 & 0 & 0 & 0 & st \\
0 & 0 & 0 &  -st & 0 \\
0 & 0 & 0 & 0 & 0 \\
0 & 0 & 0 & 0 & 0\end{array}\right ] =stx
\end{equation*}
In particular, taking
$s =1 ,$
we obtain all elements of the form
$tx ,$
$t \neq 0$
as part of this orbit.  Since
$0$
is in an orbit by itself, this shows the orbit is precisely the set of all nonzero multiples
of
$x .$
In terms of linear combinations of root vectors, or coordinates in
$\mathfrak{n} \approx k^{4}$
, this says an element
$(a_{1} ,a_{2} ,a_{3} ,a_{4}) \in B .x$
if and only if
$a_{1} =a_{2} =a_{3} =0 ,$
and
$a_{4} \neq 0.$
In other words, this shows that the orbit is
\begin{equation*}B .x =B .x_{\alpha _{1} +2\alpha _{2}} =Z(X_{\alpha _{1}} ,X_{\alpha _{2}} ,X_{\alpha _{1} +\alpha _{2}}) \cap V(X_{\alpha _{1} +2\alpha _{2}})
\end{equation*}

These are the defining equations of this orbit.  Clearly, its closure is
$\overline{B .x} =Z(X_{\alpha _{1}} ,X_{\alpha _{2}} ,X_{\alpha _{1} +\alpha _{2}}) ,$
the intersection of three coordinate hyperspaces in
$n \approx k^{4}$
which is precisely the axis of the fourth coordinate
$X_{\alpha _{1} +2\alpha _{2}}$. In particular, it is obvious that this orbit is
$1$-dimensional (it is dense in the high root space
$\mathfrak{g}_{\alpha _{1} +2\alpha _{2}}$.)

 In order to save space, we will eschew writing out the matrices from this point on,
in favor of writing elements of
$\mathfrak{n}$
as linear combinations of root vectors (or as ordered quadruples in
$k^{4})$
, and elements of
$B$
as products of elements in
$T$
and elements of the four one-dimensional root groups
$U_{\alpha }$
for
$\alpha  \in \Phi ^{ +} .$
Thus, the above calculation could be more compactly written as:
\begin{equation*}T(s ,t)U .x =T(s ,t) .x =stx
\end{equation*}
leaving the reader to check the actual matrix calculation.  In subsequent calculations, it
will be helpful to remember how the unipotent root groups act on weight vectors in rational
$G$
-modules:

\begin{lemma}
 ([H2], Proposition 27.2) Let
$\alpha  \in \Phi  ,$
and let
$v \in V_{\lambda }$
be a weight vector in any rational
$G$-module. Then each element
$u \in U_{\alpha }$
acts on
$v$
as follows:
$u .v =v +\sum _{k >0}v_{\lambda  +k\alpha } ,$
where
$v_{\lambda  +k\alpha }$
is a weight vector of weight
$\lambda  +k\alpha  ,$
and
$k$
is a positive integer.
\label{unipotentaction}
\end{lemma}

 Next, consider the orbit of the root vector
$x =x_{\alpha _{1} +\alpha _{2}}$
for the high short root
$\alpha _{1} +\alpha _{2} .$
If
$\gamma $
is a positive root, then by (Lemma \vref{unipotentaction}),
$U_{\gamma }(r) .x =x +w ,$
where
$w$
is a sum of root vectors for roots of the form
$(\alpha _{1} +\alpha _{2}) +k\gamma  ,$
for some
$k >0$
, but there are no roots of this form unless
$k =1$
and
$\gamma  =\alpha _{2} .$
It follows that
$w =0$
for all positive roots
$\gamma $
except
$\gamma  =\alpha _{2} .$
In particular,
$U_{\alpha _{1}} ,$
$U_{\alpha _{1} +\alpha _{2}} ,$
and
$U_{\alpha _{1} +2\alpha _{2}}$
all fix the vector
$x ,$
whence
$U .x =U_{\alpha _{2}} .x .$
Therefore, we have
$B .x =TU .x =TU_{\alpha _{2}} .x .$
Now take arbitrary elements
$T(s ,t) \in T$
and
$U_{\alpha _{2}}(r)$
and compute directly:
\label{shortrootorbit}
\begin{equation}T(s ,t)U_{\alpha _{2}}(r) .x =sx +rstx_{\alpha _{1} +2\alpha _{2}} . \label{shortrootorbit}
\end{equation}

 It follows, since
$s \neq 0$
that the
$x =x_{\alpha _{1} +\alpha _{2}}$
coordinate is nonzero, while it is clear that for all
$r \in k ,$
and
$s ,t \in k -\{0\} ,$
that the
$x_{\alpha _{1}}$
and
$x_{\alpha _{2}}$
coordinates are
$0 ,$
whence
\begin{equation*}B .x =T .U_{\alpha _{2}} .x \subseteq Z(X_{\alpha _{1}} ,\thinspace X_{\alpha _{2}}) \cap V(X_{\alpha _{1} +\alpha _{2}})
\end{equation*}

 To check the reverse containment, we start with an arbitrary element of the locally
closed set on the right, and find an element of
$B$
that carries
$x$
to that element.  so let
$y \neq 0 ,$
and
$z \in k$
be arbitrary.  Then the element
$(0 ,0 ,y ,z) =yx +zx_{\alpha _{1} +2\alpha _{2}}$
is an arbitrary element of our locally closed set.  But substitute the element
$T(y ,1)U_{\alpha _{2}}(\genfrac{(}{)}{}{}{z}{y})$
directly into (\vref{shortrootorbit}) to obtain:
\begin{equation*}T(y ,1)U_{\alpha _{2}}\genfrac{(}{)}{}{}{z}{y} .x =yx +y \cdot 1 \cdot \frac{z}{y}x_{\alpha _{1} +2\alpha _{2}} =(0 ,0 ,y ,z)
\end{equation*}

 This shows the reverse containment so proves the orbit is
$B .x =B .x_{\alpha _{1} +\alpha _{2}} =Z(X_{\alpha _{1}} ,\thinspace X_{\alpha _{2}})$
$ \cap $
$V(X_{\alpha _{1} +\alpha _{2}}) .$

 Next consider the orbit of
$x =x_{\alpha _{2}} .$
By lemma \vref{unipotentaction},
we only need consider the action of
$U_{\alpha _{1}}$
and
$U_{\alpha _{1} +\alpha _{2}} .$
By direct calculation, we have:
\label{alpha2orbit}
\begin{equation}T(p ,q)U_{\alpha _{1}}(s)U_{\alpha _{1} +\alpha _{2}}(r) .x =qx +psx_{\alpha _{1} +\alpha _{2}} -pqrx_{\alpha _{1} +2\alpha _{2}} =(0 ,q ,ps , -pqr) \label{alpha2orbit}
\end{equation}

 Since
$q \neq 0 ,$
this shows
$B .x \subseteq Z(X_{\alpha _{1}}) \cap V(X_{\alpha _{2}}) .$
To see we have equality we again start with an arbitrary element
$(0 ,w ,y ,z) \in Z(X_{\alpha _{1}}) \cap V(X_{\alpha _{2}})$
(so
$w \neq 0) ,$
and exhibit an element of
$B$
which carries
$x =(0 ,1 ,0 ,0)$
to it.  One such element is
$T(1 ,w)U_{\alpha _{1}}(y)U_{\alpha _{1} +\alpha _{2}}( -\frac{z}{w})$
. Indeed by (\vref{alpha2orbit}) we
obtain:
\begin{equation*}T(1 ,w)U_{\alpha _{1}}(y)U_{\alpha _{1} +\alpha _{2}}( -\frac{z}{w}) .x =wx +yx_{\alpha _{1} +\alpha _{2}} +zx_{\alpha _{1} +2\alpha _{2}} =(0 ,w ,y ,z)
\end{equation*}

 We have shown that
$B .x =B .x_{\alpha _{2}} =Z(X_{\alpha _{1}}) \cap V(X_{\alpha _{2}}) .$

 Next consider the orbit of
$x =x_{\alpha _{1}} .$
By lemma \vref{unipotentaction},
the
$U$-orbit of
$x$
is the same as the
$U_{\alpha _{2}}$-orbit. Thus, direct calculation yields:
\label{alpha1orbit}
\begin{equation}T(s ,t)U_{\alpha _{2}}(r) .x =\frac{s}{t}x -rsx_{\alpha _{1} +\alpha _{2}} -\frac{r^{2}st}{2}x_{\alpha _{1} +2\alpha _{2}} =(\frac{s}{t} ,0 , -rs , -\frac{r^{2}st}{2}) \label{alpha1orbit}
\end{equation}

 Note that because of the
$2$
in the denominator, we must avoid characteristic
$2$
fields. Since
$\frac{s}{t} \neq 0 ,$
this shows
$B .x \subseteq Z(X_{\alpha _{2}}) \cap V(X_{\alpha _{1}}) .$
However, unlike the above orbits, we do not have an equality in this case due to algebraic
dependence relations among the coordinates. Indeed, if we set
\begin{equation*}(\frac{s}{t} ,0 , -rs , -\frac{r^{2}st}{2}) =(w ,0 ,y ,z)
\end{equation*}

 Observe that
$y^{2} +2wz =0$
for every element of this form.  This shows that, in fact,
$B .x \subseteq Z(X_{\alpha _{2}} ,X_{\alpha _{1} +\alpha _{2}}^{2} +2X_{\alpha _{1}}X_{\alpha _{1} +2\alpha _{2}}) \cap V(X_{\alpha _{1}}) .$
We now claim we have an equality.  Indeed an arbitrary element of this locally closed set has
the form
$(w ,0 ,y ,z)$
with
$y^{2} +2wz =0$
and
$w \neq 0$
But it follows from (\vref{alpha1orbit}) that
\begin{align*}T(w ,1)U_{\alpha _{2}}( -\frac{y}{w}) .x =wx +yx_{\alpha _{1} +\alpha _{2}} -\frac{y^{2}}{2w}x_{\alpha _{1} +2\alpha _{2}} \\
 =(w ,0 ,y , -\frac{y^{2}}{2w}) =(w ,0 ,y ,z)\end{align*}
the last equality because
$y^{2} +2wz =0.$
This shows that
$B .x =B .x_{\alpha _{1}} =Z(X_{\alpha _{2}} ,X_{\alpha _{1} +\alpha _{2}}^{2} +2X_{\alpha _{1}}X_{\alpha _{1} +2\alpha _{2}}) \cap V(X_{\alpha _{1}})$
.

So far we have determined the orbits of the four root vectors, but taken together they
do not exhaust all of
$\mathfrak{n} .$
The remaining orbits can be taken to be be orbits of certain sums of root vectors.  For
example, consider the element
$x =x_{\alpha _{1}} +x_{\alpha _{1} +2\alpha _{2}} .$
All of
$U$
fixes
$x$
except for
$U_{\alpha _{2}}$
by lemma \vref{unipotentaction}.
By direct computation we have:
\begin{equation*}T(s ,t)U_{\alpha _{2}}(r) .x =\frac{s}{t}x_{\alpha _{1}} -rsx_{a_{1} +\alpha _{2}} +st(1 -\frac{r^{2}}{2})x_{\alpha _{1} +2\alpha _{2}} =(\frac{s}{t} ,0 ,rs ,st(1 -\frac{r^{2}}{2}))
\end{equation*}

 Now
$\frac{s}{t} \neq 0 ,$
so the orbit is contained in
$Z(X_{\alpha _{2}}) \cap V(X_{\alpha _{1}}) .$
But also
\begin{equation*}X_{\alpha _{1} +\alpha _{2}}^{2} +2X_{\alpha _{1}}X_{\alpha _{1} +2\alpha _{2}} =( -rs)^{2} +2\genfrac{(}{)}{}{}{s}{t}(st(1 -\frac{r^{2}}{2})) =2s^{2} \neq 0
\end{equation*}

 So
$B .x \subseteq Z(X_{\alpha _{2}}) \cap V(X_{\alpha _{1}}) \cap V(X_{\alpha _{1} +\alpha _{2}}^{2} +2X_{\alpha _{1}}X_{\alpha _{1} +2\alpha _{2}}) .$
We now prove the reverse containment.  Note that an arbitrary element of this locally closed
set has the form
$(w ,0 ,y ,z)$
with
$w \neq 0$
, and
$y^{2} +2wz \neq 0.$
Since
$k$
is not characteristic
$2 ,$
the element
$\frac{y^{2} +2wz}{2}$
exists and is nonzero in
$k ,$
and since
$k$
is algebraically closed, it's square root exists in
$k$
and is also nonzero.  Now by direct calculation we have
\begin{equation*}T\left (\sqrt{\frac{y^{2} +2wz}{2}} ,\frac{1}{w}\sqrt{\frac{y^{2} +2wz}{2}}\right )U_{\alpha _{2}}\left ( -y\sqrt{\frac{2}{y^{2} +2wz}}\right ) .x =(w ,0 ,y ,z)
\end{equation*}

 This proves
$B .x =B .(x_{\alpha _{1}} +x_{\alpha _{1} +\alpha _{2}}) =Z(X_{\alpha _{2}}) \cap V(X_{\alpha _{1}}) \cap V(X_{\alpha _{1} +\alpha _{2}}^{2} +2X_{\alpha _{1}}X_{\alpha _{1} +2\alpha _{2}})$
.

 The last orbit we need to consider is the orbit of
$x =x_{\alpha _{1}} +x_{\alpha _{2}} .$
Only
$U_{\alpha _{1} +2\alpha _{2}}$
fixes
$x ,$
so we need to see how all three of the other root groups act.  By direct matrix calculation
we have
\label{regularorbit}
\begin{equation}T(s ,t)U_{\alpha _{1}}(a)U_{\alpha _{2}}(b)U_{\alpha _{1} +\alpha _{2}}(c) .x =(\frac{s}{t} ,t ,(a -b)s , -st(\frac{b^{2}}{2} +c)) \label{regularorbit}
\end{equation}

 Since
$s ,t \neq 0$
, we have
$B .x \subseteq V(X_{\alpha _{1}}) \cap V(X_{\alpha _{2}}) .$
We now show the reverse containment,  Let
$(w ,u ,y ,z) \in V(X_{\alpha _{1}}) \cap V(X_{\alpha _{2}})$
be arbitrary (so
$w ,u \neq 0) .$
Then using (\vref{regularorbit}) , we
have
\begin{equation*}T(wu ,u)U_{\alpha _{1}}\genfrac{(}{)}{}{}{y}{wu}U_{\alpha _{2}}(0)U_{\alpha _{1} +\alpha _{2}}( -\frac{z}{wu^{2}}) .x =(w ,u ,y ,z)
\end{equation*}

 This shows
$B .x =V(X_{\alpha _{1}}) \cap V(X_{\alpha _{2}}) ,$
an open, dense orbit in
$\mathfrak{n} ,$
called the \emph{regular} orbit.  We are nearly finished with the proof of
our main result:

\begin{theorem}
 Let
$G$
be
$SO_{5}(k) ,$
where
$k$
is algebraically closed of characteristic not
$2 ,$
and let
$B$
be a Borel subgroup acting on
$\mathfrak{n}$
via the Adjoint action.  Then
$B$
has just seven orbits as indicated in the following table along with their defining
equations.  The dimensions of these orbits are also indicated in the table, and the closure order is
indicated by the Hasse diagram following the table.
\label{MainResult}

\begin{center}
\begin{tabular}[c]{|c|c|c|}\hline
element
$x$
of
$\mathfrak{n}$
& defining equations for
$B .x$
&
$\dim B .x$
\\
\hline
\hline
$0$
&
$Z(X_{\alpha _{1}} ,X_{\alpha _{2}} ,X_{\alpha _{1} +\alpha _{2}} ,X_{\alpha _{1} +2\alpha _{2}})$
&
$0$
\\
\hline
$x_{\alpha _{1} +2\alpha _{2}}$
&
$Z(X_{\alpha _{1}} ,X_{\alpha _{2}} ,X_{\alpha _{1} +\alpha _{2}}) \cap V(X_{\alpha _{1} +2\alpha _{2}})$
&
$1$
\\
\hline
$x_{\alpha _{1} +\alpha _{2}}$
&
$Z(X_{\alpha _{1}} ,X_{\alpha _{2}}) \cap V(X_{\alpha _{1} +\alpha _{2}})$
&
$2$
\\
\hline
$x_{\alpha _{1}}$
&
$Z(X_{\alpha _{2}} ,X_{\alpha _{1} +\alpha _{2}}^{2} +2X_{\alpha _{1}}X_{\alpha _{1} +2\alpha _{2}}) \cap V(X_{\alpha _{1}})$
&
$2$
\\
\hline
$x_{\alpha _{2}}$
&
$Z(X_{\alpha _{1}}) \cap V(\alpha _{2})$
&
$3$
\\
\hline
$x_{\alpha _{1}} +x_{\alpha _{1} +2\alpha _{2}}$
&
$Z(X_{\alpha _{2}}) \cap V(X_{\alpha _{1}}) \cap V(X_{\alpha _{1} +\alpha _{2}}^{2} +2X_{\alpha _{1}}X_{\alpha _{1} +2\alpha _{2}})$
&
$3$
\\
\hline
$x_{\alpha _{1}} +x_{\alpha _{2}}$
&
$V(X_{\alpha _{1}}) \cap V(X_{\alpha _{2}})$
&
$4$
\\
\hline
\end{tabular}\end{center}\par
In the Hasse diagram below of the closure ordering, each orbit is indicated by its
representative element from the first column of the table above:

\begin{center}\includegraphics[ width=2.84375in,]{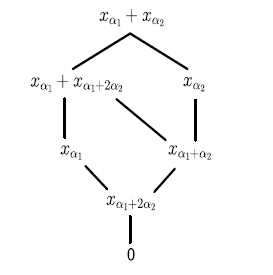}
\end{center}\par
\begin{center} The Closure order for nilpotent
$B$-orbits in type
$B_{2}$\medskip \end{center}\par
\end{theorem}

\begin{proof}
 We have already verified the entries in the first two columns of the table.  Note
that the orbit closures are just the closed sets from the defining equations.  For example, since
$B .x_{\alpha _{1} +\alpha _{2}} =Z(X_{\alpha _{1}} ,X_{\alpha _{2}}) \cap V(X_{\alpha _{1} +\alpha _{2}})$
, we have
$\overline{B .x_{\alpha _{1} +\alpha _{2}}}$
$ =Z(X_{\alpha _{1}} ,X_{\alpha _{2}}) .$
Using the closures, we can easily determine the dimensions in the third column as well as
the closure ordering.  Note that for polynomials
$f_{i}$
in
$r$
variables, the dimension of
$Z(f_{1} ,f_{2} , . . . ,f_{k})$
is just
$r -k$
provided that the
$f_{i}$
are all algebraically independent.  It should be clear that we found all the algebraic
dependencies when we worked out the defining equations, so that the
$f_{i}$
are algebraically independent in the closed sets in the second column of the table.  Thus,
since
$r =\dim \mathfrak{n} =4$
, the dimensions in the third column are equal to
$4 -k ,$
where
$k$
is the number of polynomials whose zero sets define the orbit closure.

The only nontrivial
containment for the closure ordering is
$B .x_{a_{1} +2\alpha _{2}} \subseteq \overline{B .x_{\alpha _{1}}}$
, which happens if and only if
$\overline{B .x_{\alpha _{1} +2\alpha _{2}}} \subseteq \overline{B .x_{\alpha _{1}}} .$
So take an arbitrary element
$x \in \overline{B .x_{\alpha _{1} +2\alpha _{2}}} =Z(X_{\alpha _{1}} ,X_{\alpha _{2}} ,X_{\alpha _{1} +\alpha _{2}}) .$
Then, since both
$X_{\alpha _{1}} =0$
and
$X_{\alpha _{1} +\alpha _{2}} =0 ,$
it follows that both
$X_{\alpha _{1} +\alpha _{2}}^{2} =0$
and
$X_{\alpha _{1}}X_{\alpha _{1} +2\alpha _{2}} =0$
when evaluated at
$x .$
Therefore,
$X_{\alpha _{1} +\alpha _{2}}^{2} +2X_{\alpha _{1}}X_{\alpha _{1} +2\alpha _{2}} =0$
as well, so
$x \in Z(X_{\alpha _{2}} ,X_{\alpha _{1} +\alpha _{2}}^{2} +2X_{\alpha _{1}}X_{\alpha _{1} +2\alpha _{2}}) =_{_{}}\overline{B .x_{\alpha _{1}}}$
, showing the desired containment.  The other containments shown in the Hasse diagram
follow similarly.

All that remains to show is that we have exhausted all the nilpotent orbits in $\mathfrak{n}$.  So let $n =wx_{\alpha _{1}} +xx_{\alpha _{2}} +yx_{\alpha _{1} +\alpha _{2}} +zx_{\alpha _{1} +2\alpha _{2}} =(w ,x ,y ,z)$ be an arbitrary element of $\mathfrak{n} .$  We must show $n$ lies in one of these seven orbits.  We will distinguish cases according to how many and which of the four coordinates are $0.$  If both $w$ and $x$ are nonzero, then $n \in V(X_{\alpha _{1}}) \cap V(X_{\alpha _{2}}) =B .(x_{\alpha _{1}} +x_{\alpha _{2}}) ,$ the regular orbit.  So it only remains to consider cases when one or both of $w ,x$ are $0.$ First, suppose $w =0$ but $x \neq 0.$  Then $n =(0 ,x ,y ,z) \in Z(X_{\alpha _{1}}) \cap V(X_{\alpha _{2}}) =B .x_{\alpha _{2}} .$ On the other hand, suppose $w \neq 0$ and $x =0 ,$ so $n =(w ,0 ,y ,z) \in Z(X_{\alpha _{2}}) \cap V(X_{\alpha _{1}}) .$  But then $n \in B .(x_{\alpha _{1}} +x_{\alpha _{1} +2\alpha _{2}})$ or $n \in B .x_{\alpha _{1}} ,$ depending on whether or not $y^{2} +2wz =0.$

Lastly, we consider cases where
$w =0 =x .$
In this case,
$n =(0 ,0 ,y ,z) \in Z(X_{\alpha _{1}}) \cap Z(X_{\alpha _{2}}) .$
If
$y \neq 0 ,$
then
$n \in Z(X_{\alpha _{1}} ,X_{\alpha _{2}}) \cap V(X_{\alpha _{1} +\alpha _{2}}) =B .x_{\alpha _{1} +\alpha _{2}} .$
On the other hand, if
$y =0 ,$
then
$n =(0 ,0 ,0 ,z)$
which belongs to either
$B.0$
or
$B .x_{\alpha _{1} +2\alpha _{2}}$
, depending on whether or not
$z =0.$
This covers all possible cases, and in each case,
$n$
was in one of the above mentioned orbits, whence the union of the
$7$
orbits is all of
$\mathfrak{n} .$
\end{proof}

This completes the proof. \bigskip

\section{
Conclusions.}
The result that there are only finitely many nilpotent $B$-orbits for $SO_{5}(k)$ can be phrased in terms of a general concept for algebraic group actions called modality (see [PR], for example).  

Let
$G$
be an arbitrary algebraic group acting morphically on a non-empty variety
$V .$
The \emph{modality} of the action is
\label{modality}
\begin{equation}\ensuremath{\operatorname*{mod}}(G ,V) =\underset{Z}{\max }\underset{z \in Z}{\thinspace \min }codi\text{}m_{Z}\text{
}(G^{0} .z) , \label{modality}
\end{equation}
where
$Z$
runs through all irreducible
$G^{0}$-invariant subvarieties of
$V$. Here, $G^{0}$ is the connected component of the identity in $G .$
Informally, the modality is the maximum number of (continuous) parameters on which a family
of
$G$
orbits may depend.

Although we are mainly interested in nilpotent orbits for a Borel subgroup
$B$
of
$G ,$
much of the literature is written in terms of the more general case of orbits for a
\emph{parabolic} subgroup
$P ,$
which is any closed subgroup containing a Borel subgroup.  If
$P$
is parabolic$ ,$
denote its Lie algebra by
$\mathfrak{p} ,$
and the nilradical of
$\mathfrak{p}$
by
$\mathfrak{n}(\mathfrak{p}) .$
Then
$P$
acts on
$\mathfrak{n}(\mathfrak{p})$
via the Adjoint representation, and the \emph{modality of}
$P$
is defined to be
$\ensuremath{\operatorname*{mod}}(P ,\mathfrak{n}(\mathfrak{p}))$. Thus the modality of
$P$
is
$0$
precisely when there are only finitely many nilpotent
$P$-orbits in
$\mathfrak{n}(\mathfrak{p}) .$

When
$P =G ,$
the nilradical of
$\mathfrak{g}$
is trivial since
$\mathfrak{g}$
is simple, so the modality of
$G$
is trivially
$0.$
At the other extreme, the modality of
$B$
is almost never
$0.$
So one consequence of theorem \vref{MainResult} is that if $p \neq 2 ,$ then $B$ has modality $0$ in type $B_{2} .$  Of course, this is a well-known result. Based on earlier work in [BH], In [K], all the Borel subgroups of modality zero were determined in characteristic zero:

\begin{theorem}
(Kashin [K], 1990) Let
$G$
be quasi-simple over
$k$
, where
$k$
has characteristic zero, and suppose
$B$
is a Borel subgroup of
$G .$
The number of orbits of
$B$
on
$\mathfrak{n}$
is finite (that is,
$B$
has modality zero) if and only if
$G$
is type
$A_{n}$
for
$n \leq 4 ,$
or
$G$
is type
$B_{2} .$
\label{Kashin}
\end{theorem}

Aside from the consequences of this theorem for our investigation on nilpotent $B$-orbits, Kashin's result launched an investigation into the modality of parabolic subgroups in general.  For example, see [BHRZ],[HR],[P],[PR],[R1], and [R2].  In fact, Theorem 1.1 in [HR] shows there is a
strong connection between the modality of a parabolic subgroup and the length of a descending
central series of $R_{u}(P) ,$ the unipotent radical of $P ,$
also called the \emph{nilpotency class }of $R_{u}(P) .$
Using this theorem, one can easily recover Kashin's original theorem, with the added
benefit that the proof is valid in good prime characteristics for
$G$
as well as for characteristic zero.
In type $A ,$ all primes are good, and in type $B ,$ all primes are good except $p =2.$

In a previous version of
this paper, using similar techniques as here, we showed
directly that there are finitely many nilpotent
$B$
orbits for
$G$
of type
$A_{1} ,A_{2} ,A_{3}$
and
$A_{4}$
without any restrictions on
$p$, and used that information to determine the dimensions of the orbits and the closure ordering. A referee pointed out to us that the closure orderings for the
$4$
type
$A$
cases were already discussed in [BHRZ], making a lot of our work seem redundant.  Note that
the techniques used in [BHRZ] are quite different than ours - they are much
more sophisticated than our matrix calculations. Their approach has some advantages, such as both being more
elegant than our approach and also closer in spirit to the way nilpotent
$G$-orbits are classified.  A possible advantage of our techniques, though, is that they yield the explicit polynomial defining equations of each orbit.  It may be an advantage to knowing these defining equations in applying this work,
perhaps to computing support varieties of baby Verma modules as discussed in Section 1, or perhaps
for other applications.  For this reason, we have uploaded our type $A$ calculations [BV]
to  arxiv.org, so that despite the overlap with [BHRZ], our tables for these orbits are publicly available.
Here we conclude by simply reminding the reader how many orbits there are in each case:
$2$
orbits in type
$A_{1} ,$
$5$
orbits in type
$A_{2} ,$
$16$
orbits in type
$A_{3} ,$
and
$61$
orbits in type
$A_{4} .$
For the details of the defining equations, etc., consult [BV], and for the dimensions of
each orbit and the Hasse diagrams of the closure order in these cases, valid for all characteristics, consult either [BHRZ] or
[BV]. \bigskip
\begin{center}{\Large References}\bigskip \end{center}\par
[B] A. Borel, \textit{Linear Algebraic Groups}, Second Enlarged Edition,
Springer, Graduate Texts in Mathematics \textbf{126}, 1991 \bigskip

[BHRZ] T. Brustle, L. Hille, G. Rohrle, and G, Zwara, \textit{The Bruhat-Chevalley
Order of Parabolic Group Actions in General Linear Groups and Degeneration for }
$\Delta $
\textit{-Filtered Modules}, Advances in Mathematics, \textbf{148}, pp.
203-242, 1999. \bigskip

[Bou] N. Bourbaki, \textit{Elements of Mathematics, Lie Groups and Lie Algebras,
Chapters}\textit{4-6}, Springer, 2002 \bigskip

[BH] H. B{\"u}rgstein and W. H. Hesselink, \textit{Algorithmic Orbit Classication for
Some Borel Group Actions}, Composito Mathematica\textbf{ 61}, pp 3-41, 1987 \bigskip

[BV] M. Burkhart and D. Vella, \textit{Defining Equations for Nilpotent Orbits for
Borel Subgroups of Modality Zero in Type $A_{n}$,} http://arxiv.org/alg/.1978875 \bigskip

 [C] R. W. Carter, \textit{Finite Groups of Lie Type}, Wiley Classics
Library, 1993\bigskip

[CM] D. H. Collingwood and W. M. McGovern, \textit{Nilpotent Orbits in Semisimple
Lie Algebras}, Van Nostrand Rheinhold Mathematics Series, 1993 \bigskip

[FP] E. M. Friedlander and B. J. Parshall, \textit{Support Varieties for Restricted
Lie Algebras}, Invent. Math. \textbf{86}, pp. 553-562, 1986 \bigskip

[H1] J. E. Humphreys,\textit{ Introduction to Lie Algebras and Representation
Theory}, Springer, Graduate Texts in Mathematics \textbf{9}, 1972 \bigskip

[H2] J.E. Humphreys,\textit{ Linear Algebraic Groups}, Springer, Graduate
Texts in Mathematics \textbf{21}, 1975 \bigskip

 [HR] L. Hille and G. R{\"o}hrle\textit{, A Classification of Parabolic Subgroups of
Classical Groups with a Finite Number of Orbits on the Unipotent Radical}, Transformation
Groups, Vol. 4, No. 1, pp. 35-52, 1999 \bigskip

 [J] J. Jantzen, \textit{Nilpotent Orbits in Representation Theory}, in
Lie Theory - Lie Algebras and Representations, Springer Progress in Mathematics \textbf{228,}
pp. 1-211, 2004. \bigskip

 [K] V. V. Kashin, \textit{Orbits of adjoint and coadjoint actions of Borel
subgroups of semisimple algebraic groups}. (In Russian) Questions of Group Theory and
Homological Algebra, Yaroslavl, 141-159, 1990 \bigskip

[NPV] D. K. Nakano, B. J. Parshall and D. C. Vella, \textit{Support Varieties for
Algebraic Groups}, Journal fur die reine und angewandte Mathematik, \textbf{547}, pp.
15-49, 2002 \bigskip

\bigskip [P] V. L. Popov, \textit{A Finiteness Theorem
for Parabolic Subgroups of Fixed Modality}, Indag. Mathem., N.S. \textbf{8} (1),
pp.125-132, 1997 \bigskip

[PR] V. L. Popov and G. R{\"o}hrle, \textit{On the Number of Orbits of a Parabolic
Subgroup on its Unipotent Radical}, Algebraic Groups and Lie Groups, G. I. Lehrer, Ed.,
Australian Mathematical Society, Vol. \textbf{9}, Cambridge University Press, 1997 \bigskip

 [R1] G. R{\"o}hrle, \textit{Parabolic Subgroups of Positive Modality},
Geometriae Dedicata,\textbf{ 60}, pp.163-186, 1996 \bigskip

 [R2] G. R{\"o}hrle,
\textit{On the Modality of Parabolic
Subgroups of Linear Algebraic Groups}, Manuscripta Math. 98, pp, 9-20, 1999 \bigskip
\end{document}